\documentclass{article}

\usepackage{amsmath,amsfonts,amssymb}
\usepackage{graphicx}
\usepackage{dsfont}
\usepackage{bbold}
\usepackage{amsthm}
\usepackage[utf8]{inputenc}
\usepackage{appendix}
\usepackage{geometry}
\newtheorem{thm}{Theorem}

\newtheorem{defn}[thm]{Definition}
\newtheorem{prp}{Proposition}
\newtheorem{claim}{Claim}
\newtheorem{lemma}{Lemma}

\title{\textbf{Application of Perron Trees to Geometric Maximal Operators}}

\author{Anthony Gauvan\footnote{Institut Mathématiques d'Orsay, Facultés des Sciences, 91400 Orsay}}

\begin{document}

\maketitle

\begin{abstract}
We characterize the $L^p(\mathbb{R}^2)$ boundeness of the geometric maximal operator $M_{a,b}$ associated to the basis $\mathcal{B}_{a,b}$ ($a,b > 0$) which is composed of rectangles $R$ whose eccentricity and orientation is of the form $$\left( e_R ,\omega_R \right) = \left( \frac{1}{n^a} , \frac{\pi}{4n^b} \right)$$ for some $n \in \mathbb{N}^*$. The proof involves \textit{generalized Perron trees}, as constructed in \cite{KATHRYN JAN}.
\end{abstract}

\section{Introduction}\label{S:1}

In \cite{BATEMAN}, Bateman has concluded the study of \textit{directional maximal operators in the plane}. In this text, we study \textit{geometric maximal operators} which are a natural generalisation of the directional operators. However, their study requires a precise understanding of the correlation between the \textit{eccentricity} and the \textit{orientation} of families of rectangles. In this text, we are able to prove sharp results concerning the $L^p(\mathbb{R}^2)$ range of boundedness of geometric maximal operators whose parameters vary in a polynomial way.

\subsection*{Definitions}

We work in the euclidean plane $\mathbb{R}^2$ ; if $A$ is a measurable subset of $\mathbb{R}^2$ we denote by $|A|$ its two dimensional Lebesgue measure. We denote by $A \sqcup B$ the union of $A$ and $B$ when $|A\cap B|=0$.

Denote by $\mathcal{R}$ the collection containing all rectangles of $\mathbb{R}^2$ ; for $R \in \mathcal{R}$ we define its \textit{orientation} as  the angle $\omega_R \in [0,\pi) $ that its longest side makes with the $Ox$-axis and its \textit{eccentricity} as the ratio $e_R \in (0,1]$ of its shortest side by its longest side.

For an arbitrary non empty family $\mathcal{B}$ contained in $\mathcal{R}$, we define the associated \textit{derivation basis} $\mathcal{B}^*$ by $$\mathcal{B}^* = \left\{ \Vec{t} + h R : \Vec{t} \in \mathbb{R}^2, h>0, R \in \mathcal{B} \right\}.$$ The derivation basis $\mathcal{B}^*$ is simply the smallest collection which is invariant by dilation and translation and that contains $\mathcal{B}$. Without loss of generality, we identify the derivation basis $\mathcal{B}^*$ and any of its generator $\mathcal{B}$.

Our object of interest will be the \textit{geometric maximal operator $M_\mathcal{B}$ generated by $\mathcal{B}$} which is defined as $$M_\mathcal{B}f(x) := \sup_{x \in R \in \mathcal{B}^*}  \frac{1}{|R|} \int_R |f|$$ for any $f \in L_{loc}^{1}(\mathbb{R}^2)$ and $x \in \mathbb{R}^2$. Observe that the upper bound is taken on elements of $\mathcal{B}^*$ that contain the point $x$. The definitions of $\mathcal{B}^*$ and $M_\mathcal{B}$ remain valid when we consider that $\mathcal{B}$ is an arbitrary family composed of open bounded convex sets. For example in this note, for technical reasons and without loss of generality, we will work at some point with triangles instead of rectangles.

For $p \in (1, \infty]$ we define as usual the operator norm $\|M_\mathcal{B}\|_p$ of $M_\mathcal{B}$ by $$ \|M_\mathcal{B}\|_p = \sup_{ \|f\|_p = 1} \|M_\mathcal{B}f\|_p.$$ If $\|M_\mathcal{B}\|_p < \infty$ we say that $M_\mathcal{B}$ is \textit{bounded} on $L^p(\mathbb{R}^2)$.  The boundedness of a maximal operator $M_\mathcal{B}$ is related to the geometry that the family $\mathcal{B}$ exhibits. 

\begin{defn}
We will say that the operator $M_\mathcal{B}$ is a \textbf{good operator} when it is bounded on $L^p(\mathbb{R}^2)$ for any $ p > 1$. On the other hand, we say that the operator $M_\mathcal{B}$ is a \textbf{bad operator} when it is unbounded on $L^p(\mathbb{R}^2)$ for any $1 < p < \infty$.
\end{defn}

On the $L^p(\mathbb{R}^2)$ scale, to be able to say that a operator $M_\mathcal{B}$ is good or bad is an optimal result. One can also be interested by the behavior near endpoint ($p = 1$ and $p=\infty$) but we won't consider this question here ; the reader might consult D'Aniello, Moonens and Rosenblatt \cite{DMR}, D'Aniello and Moonens  \cite{DM}-\cite{DM2} or Stokolos \cite{STOKOLOS}.

\subsection*{Directional maximal operators}

Researches have been done in the case where $\mathcal{B}$ is equal to $\mathcal{R}_\Omega :=  \left\{ R \in \mathcal{R} : \omega_R \in \Omega \right\}$ where $\Omega$ is an arbitrary set of directions in $[0,\pi)$. In other words, $\mathcal{R}_\Omega$ is the set of \textit{all} rectangles whose orientation belongs to $\Omega$. We say that $\mathcal{R}_\Omega$ is a \textit{directional basis} and to alleviate the notation we denote $$M_{\mathcal{R}_\Omega} := M_\Omega.$$ In the literature, the operator $M_\Omega$ is said to be a \textit{directional maximal operator}. The study of those operators goes back at least to Cordoba and Fefferman's article \cite{CORDOBAFEFFERMAN II} in which they use geometric techniques to show that if $\Omega = \left\{ \frac{\pi}{2^k} \right\}_{ k \geq 1}$ then $M_\Omega$ has weak-type $(2,2)$. A year later, using Fourier analysis techniques, Nagel, Stein and Wainger proved in \cite{NSW} that $M_\Omega$ is actually bounded on $L^p(\mathbb{R}^2)$ for any $p > 1$. In \cite{ALFONSECA}, Alfonseca has proved that if the set of direction $\Omega$ is a \textit{lacunary set of finite order} then the operator $M_\Omega$ is bounded on $L^p(\mathbb{R}^2)$ for any $p > 1$. Finally in \cite{BATEMAN}, Bateman proved the converse and so characterized the $L^p(\mathbb{R}^2)$-boundedness of directional operators. Precisely he proved the following theorem.

\begin{thm}[Bateman's Theorem]\label{ T : bateman }
Fix an arbitrary set of directions $\Omega \subset [0,\pi)$. The directional maximal operator $M_\Omega$ is either good or bad.
\end{thm}

Hence we know that a set of directions $\Omega$ always yields a directional operator $M_\Omega$ that is either good or bad. Merging the vocabulary, we use the following definition.

\begin{defn}
We say that a set of directions $\Omega$ is a good set of directions when $M_\Omega$ is good and that it is a  bad set of directions when $M_\Omega$ is bad.
\end{defn}

The notion of good/bad is perfectly understood for a set of directions $\Omega$ and the associated directional operator $M_\Omega$. To say it bluntly, $\Omega$ is a good set of directions if and only if it can be included in a finite union of lacunary sets of finite order. If this is not possible, then $\Omega$ is a bad set of directions ; see \cite{BATEMAN}. We now turn attention to maximal operator which are not directional.

\subsection*{Geometric maximal operators}

\textbf{In this text, we will focus on geometric maximal operator which are not directional}. We recall two results in the direction of Bateman's Theorem for an arbitrary basis $\mathcal{B}$ included in $\mathcal{R}$. The first one is a result in \cite{HAGELSTEINSTOKOLOS} where Hagelstein and Stokolos proved the following theorem.

\begin{thm}
Fix an arbitrary basis $\mathcal{B}$ in $\mathcal{R}$ and suppose that there exist constants $t_0 \in (0,1)$ and $C_0 > 1$ such that for any bounded measurable set $E \subset \mathbb{R}^2$ one has $$\left|\{M_\mathcal{B}\mathbb{1}_E > t_0 \} \right| \leq C_0\left|E \right|.$$ In this case there exists $p_0$ depending on $(t_0,C_0)$ such that for any $p > p_0$ we have $\| M_\mathcal{B}\|_p < \infty$.
\end{thm}

In \cite{GAUVAN}, we have shown that one can associate to any basis $\mathcal{B}$ included in $\mathcal{R}$ a geometric quantity denoted by $\lambda_{[\mathcal{B}]} \in \mathbb{N} \cup \{ \infty \}$ that we call the \textit{analytic split} of the family $\mathcal{B}$. We insist on the fact that the analytic split \textit{is not} defined by abstract means but really concrete ; in certain settings one can easily compute it. The analytic split of a basis allows us to controle the $p$-norm of the associated geometric maximal operator.

\begin{thm}
For any basis $\mathcal{B}$ in $\mathcal{R}$ and any $1 < p < \infty$ we have $$\log( \lambda_{[\mathcal{B}]}) \lesssim_p \|M_\mathcal{B}\|_p^p.$$
\end{thm}

This Theorem implies that any basis $\mathcal{B}$ whose analytic split is infinite yields bad maximal operators $M_\mathcal{B}$. Moreover, it is easy to exhibit a lot of bases $\mathcal{B}$ whose analytic split is infinite.

\begin{thm}
If $\lambda_{ [\mathcal{B}]} = \infty$ then $M_\mathcal{B}$ is bad.
\end{thm}

We are going to state our results now.

\subsection*{Results}\label{ S : results }

\textbf{As said earlier, we consider a family of geometric maximal operators which are \textit{not} directional maximal operators. Moreover we will always work with bases $\mathcal{B}$ such that its associated set of directions $$\Omega_{\mathcal{B}} := \{ \omega_R : R \in \mathcal{B} \}$$ is a bad set of directions.} Indeed if $\Omega_\mathcal{B}$ is a good set of directions using the trivial estimate $M_\mathcal{B} \leq M_{\Omega_{\mathcal{B}}}$ we know that $M_\mathcal{B}$ is also a good operator.

Fix two real positive numbers $a,b > 0$ and denote by $\mathcal{B}_{a,b}$ the basis of rectangles $R$ whose eccentricity and orientation are of the form $$\left( e_R ,\omega_R \right) = \left( \frac{1}{n^a} , \frac{\pi}{4n^b} \right)$$ for some $n \in \mathbb{N}^*$. We denote by $M_{a,b}$ the geometric maximal operator associated to the basis $\mathcal{B}_{a,b}$. We prove the following theorem.

\begin{thm}\label{THMA}
If $a < b$ then $M_{a,b}$ is a good operator. If not then $M_{a,b}$ is a bad operator. 
\end{thm}

We shall prove Theorem \ref{THMA} thanks to Theorems \ref{T : main 1} and \ref{T : main 2}. Denote by ${\boldsymbol{t}} =  \left\{ t_k \right\}_{k \geq 1} \subset [0,\frac{\pi}{4}]$ a sequence decreasing to $0$ and by ${\boldsymbol{e}} = \left\{ e_k \right\}_{k \geq 1} \subset (0,1]$ any positive sequence. One should consider the sequence $\boldsymbol{t}$ as a sequence of angles (or tangent of angles) that forms a bad set of directions whereas the sequence $\boldsymbol{e}$ stands for an arbitrary sequence of eccentricity. For $k \geq 1$ consider a rectangle $$R_k := R_k(\boldsymbol{e},\boldsymbol{t})$$ whose orientation and eccentricity are defined by $\left( e_{R_k} , \omega_{ R_k } \right) = \left( e_k , t_k \right)$. Define then the basis $$\mathcal{B} = \mathcal{B}(\boldsymbol{t},\boldsymbol{e})$$ as the one generated by the rectangles $\left\{ R_k \right\}_{k \geq 1}$. Our first result reads as follow.

\begin{thm}\label{T : main 1}
Suppose there is a constant $C > 0$ such that for any $k \geq 1$, $t_k\leq  C e_k$. In this case the operator $M_\mathcal{B}$ is a good operator.
\end{thm}

We define now the following quantity associated to the sequence $\boldsymbol{t}$  $$ \tau_{\boldsymbol{t}} := \sup_{k \geq 0, l \leq k} \left( \frac{t_{k+2l} - t_{k+l}}{t_{k+l} - t_{k}} + \frac{t_{k+l} - t_{k}}{t_{k+2l} - t_{k+l}} \right) \in (0,\infty].$$ This quantity yields information on the goodness/badness of the set $\{ t_k \}_{k\geq 1}$ seen as a set of directions. Indeed if $\tau_{\boldsymbol{t}}$ is finite then the set of directions $\Omega = \{ t_k \}_{k \geq 1}$ forms a bad set of directions. In some sense, this quantity indicates to which point the sequence $\boldsymbol{t}$ is \textit{uniformly distributed near} $0$. For example, the sequence $\boldsymbol{t} = \{ \frac{1}{k} \}$ look likes a uniform distribution near $0$ and we have $\tau_{\boldsymbol{t}} < \infty$. On the other hand the sequence $\boldsymbol{t} = \{ \frac{1}{2^k} \}_{k \geq 1}$ converges rapidly to $0$ and we have $\tau_{\boldsymbol{t}} = \infty$. The second result reads as follow. 

\begin{thm}\label{T : main 2}
Suppose that $\tau_{\boldsymbol{t}} < \infty$ and also that there is a constant ${\mu_0} > 0$ such that for any $k \geq 1$ we have $ e_k < {\mu_0} |t_k-t_{k+1}|$. In this case, the maximal operator $M_\mathcal{B}$ is a bad operator.
\end{thm}

Before going into the proofs, let us expose general remarks about geometric maximal operators that will be useful.

\subsection*{How can we prove that $M_\mathcal{B}$ is bad ?}

To prove that an operator $M_\mathcal{B}$ is bad, the idea is to create an exceptional geometric set adapted to the basis $\mathcal{B}$ ; precisely, one can try to find a small fixed value $0 < \eta_0 < 1$ such that for any $\epsilon > 0$ there is a subset $X$ in $\mathbb{R}^2$ satisfying $$|X| \leq \epsilon |\left\{ M_\mathcal{B}\mathbb{1}_X > \eta_0 \right\}|.$$ If this holds then for any $p > 1$ we have $$ \int (M_\mathcal{B}\mathbb{1}_X)^p \geq \eta_0^p |\left\{ M_\mathcal{B}\mathbb{1}_X > \eta_0 \right\}| \geq \eta_0^p \frac{\|\mathbb{1}_X\|_p^p}{\epsilon} $$ since $|X|^\frac{1}{p} = \|\mathbb{1}_X \|_p$. Hence for any $\epsilon > 0$ we have $ \|M_\mathcal{B}\|_p \geq \eta_0^p\epsilon^{-\frac{1}{p}} $ and $\|M_\mathcal{B}\|_p^p = \infty$ for any $1 < p < \infty$. The question remains to understand how one can find/construct such a set $X$ ? Of course this possibility depends on the basis $\mathcal{B}$. For example consider the case where $\mathcal{B}:=\mathcal{R}$ is as big as possible. The following property is true (it is a consequence of proposition \ref{ P : gpt }) : for any large constant $A > 1$ there exists a finite family of rectangles $\{ R_i \}_{i \leq m}$ in $\mathcal{R}$ satisfying $$\left| \bigcup_{i \leq m}  2R_i \right| \geq A \left| \bigcup_{i \leq m}  R_i \right|.$$ Considering then the set $X = \bigcup_{i \leq m} R_i$ it is easy to see that one has $$ \left| X \right| \leq \frac{1}{A} \left| \{ M_\mathcal{R}\mathbb{1}_X > \frac{1}{4} \} \right|$$ which implies that the maximal operator $M_\mathcal{R}$ is a bad operator. A \textit{Perron tree} (or \textit{generalized Perron tree}) formed with a basis $\mathcal{B}$ of rectangles is a concrete construction of such a set $X$ (or more precisely a sequence of sets) for any $\epsilon > 0$ and a fixed value $\eta_0$.

\subsection*{From rectangles to triangles}

\begin{figure}[h!]
\centering
\includegraphics[scale=1.5]{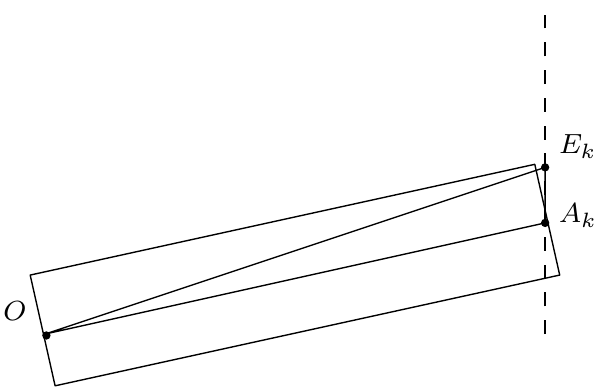}
\caption{ A rectangle $R_k$ and a triangle $T_k$, both object are oriented along $\simeq t_k$ and have an eccentricity $\simeq e_k$. }
\end{figure}

Without loss of generality, we will work at some point with triangles instead of rectangles. For any $k \geq 1$ define the triangle $T_k$ as  $$T_k := T_k(\boldsymbol{e},\boldsymbol{t}) = OA_kE_k$$ where $O = (0,0)$, $A_k= (1, t_k)$ and $E_k = (1, t_k + e_k)$. Loosely speaking, the triangle $T_k$ is a triangle which is \textit{oriented} along the direction $t_k$ and of \textit{eccentricity} $e_k$. Denoting by $\mathcal{B}'$ the basis generated by the triangles $T_k$ one can observe that we have the following property. For any $R \in \mathcal{B}$ there exists $T \in \mathcal{B}'$ satisfying for some vector $\Vec{t} \in \mathbb{R}^2$ $$\Vec{t} + \frac{1}{16}T \subset R \subset T$$ and conversely for any $T \in \mathcal{B}'$ there exists $R \in \mathcal{B}$ satisfying for some vector $\Vec{t} \in \mathbb{R}^2$ $$\Vec{t} + \frac{1}{16}R \subset T \subset R.$$ This implies that for any $f \in L_{\text{loc}}^1(\mathbb{R}^2)$ and $x \in \mathbb{R}^2$ we have $$ M_\mathcal{B}f(x) \simeq M_{\mathcal{B}'}f(x).$$ Hence it is equivalent to work with $\mathcal{B}$ or with $\mathcal{B}'$ and we will denote both basis by $\mathcal{B}$.

\section*{Acknowledgments}

I warmly thank Laurent Moonens, Emmanuel Russ and the two referees for their useful comments which have certainly improved the present text.

\section{Proof of Theorem \ref{T : main 1}}

It is well know that the operator $M_{ \{ 0\}}$ associated to the basis $\mathcal{R}_{ \{ 0\}} = \{ R \in \mathcal{R} : \omega_R = 0\}$ is a good operator. Now by easy geometric observation and using the property that $t_k < C e_k$, one can prove that for any $R \in \mathcal{B}$ there exists a rectangle $P \in \mathcal{R}_{ \{ 0\}}$ such that $$R \subset P$$ and also $$|P| \leq 8(1+C) |R|.$$ This property allows us to use the  operator $M_{ \{ 0\}}$ in order to dominate pointwise $M_\mathcal{B}$. Fix any $f \in L^1_{loc}(\mathbb{R}^2)$ and any $R \in \mathcal{B}$ and the associated rectangle $P \in \mathcal{R}_{ \{ 0 \} }$ ; we have $$\frac{1}{|R|} \int_R |f| \leq \frac{8(1+C)}{|P|} \int_P |f| $$ which shows that for any $x \in \mathbb{R}^2$ we have $$M_\mathcal{B}f(x) \leq 8(1+C) M_{ \{ 0\}}f(x).$$ The conclusion comes from the fact that the strong maximal operator $M_{ \{ 0\}}$ is a good operator.

\section{Proof of Theorem \ref{T : main 2}}\label{S:R}

The proof of Theorem \ref{T : main 2} relies on geometric estimates and the construction of generalized Perron trees.

\subsection*{Geometric estimates}\label{S:ge}

\begin{figure}[h!]
\centering
\includegraphics[scale=0.8]{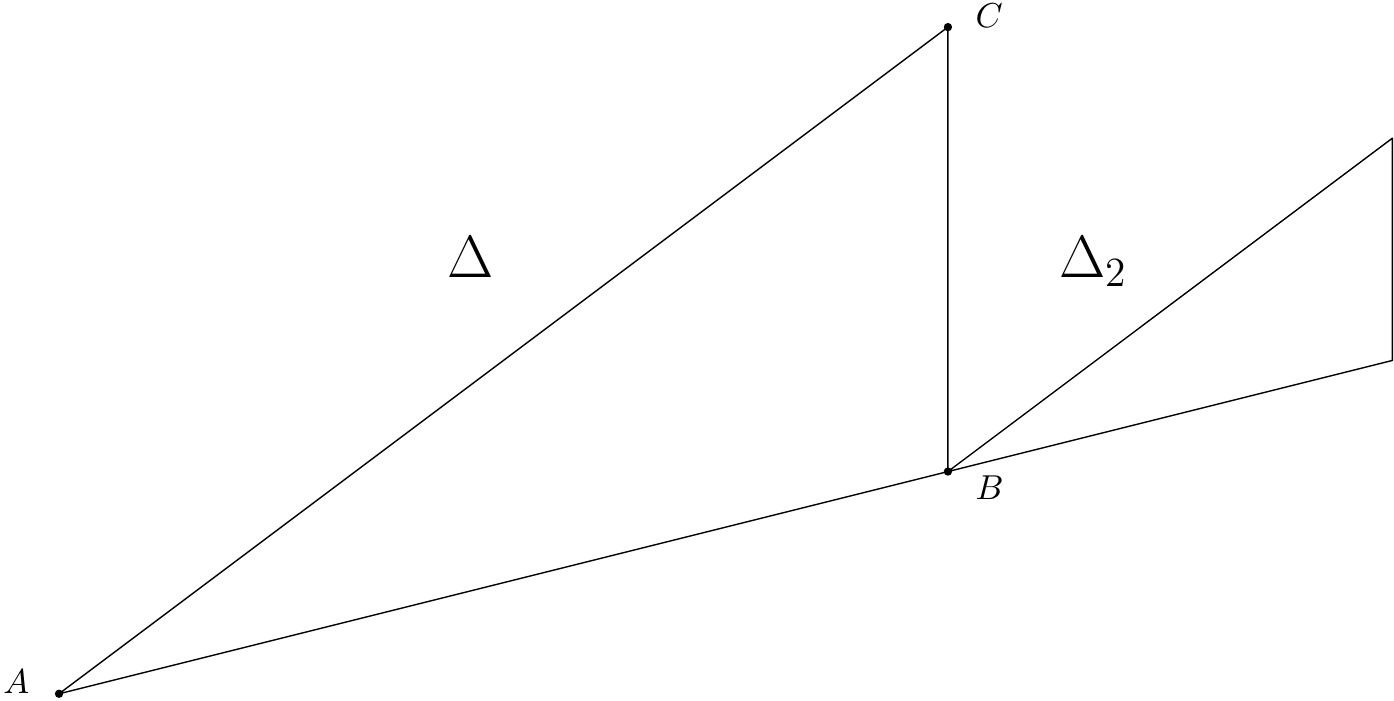}
\caption{ The triangles $\Delta$ and $\Delta_2$ will usually be in this position. }
\end{figure}

We start by establishing two geometric estimates. Fix an arbitrary open triangle $\Delta = ABC$ and consider the triangle $\Delta_2$ defined as $\Delta_2 := \Vec{B} + \frac{1}{2}(\Delta - \Vec{A})$.

\begin{lemma}[Geometric estimate I]\label{L : geo estimate 1}
The following inclusion holds $$\Delta_2 \subset \left\{ M_{ \left\{\Delta\right\} }\mathbb{1}_{\Delta} \geq \frac{1}{4} \right\}.$$ In other words, the level set  $\left\{ M_{\left\{\Delta\right\}}\mathbb{1}_{\Delta} \geq \frac{1}{4} \right\}$ contains $\Delta_2$.
\end{lemma}

\begin{figure}[h!]
\centering
\includegraphics[scale=0.6]{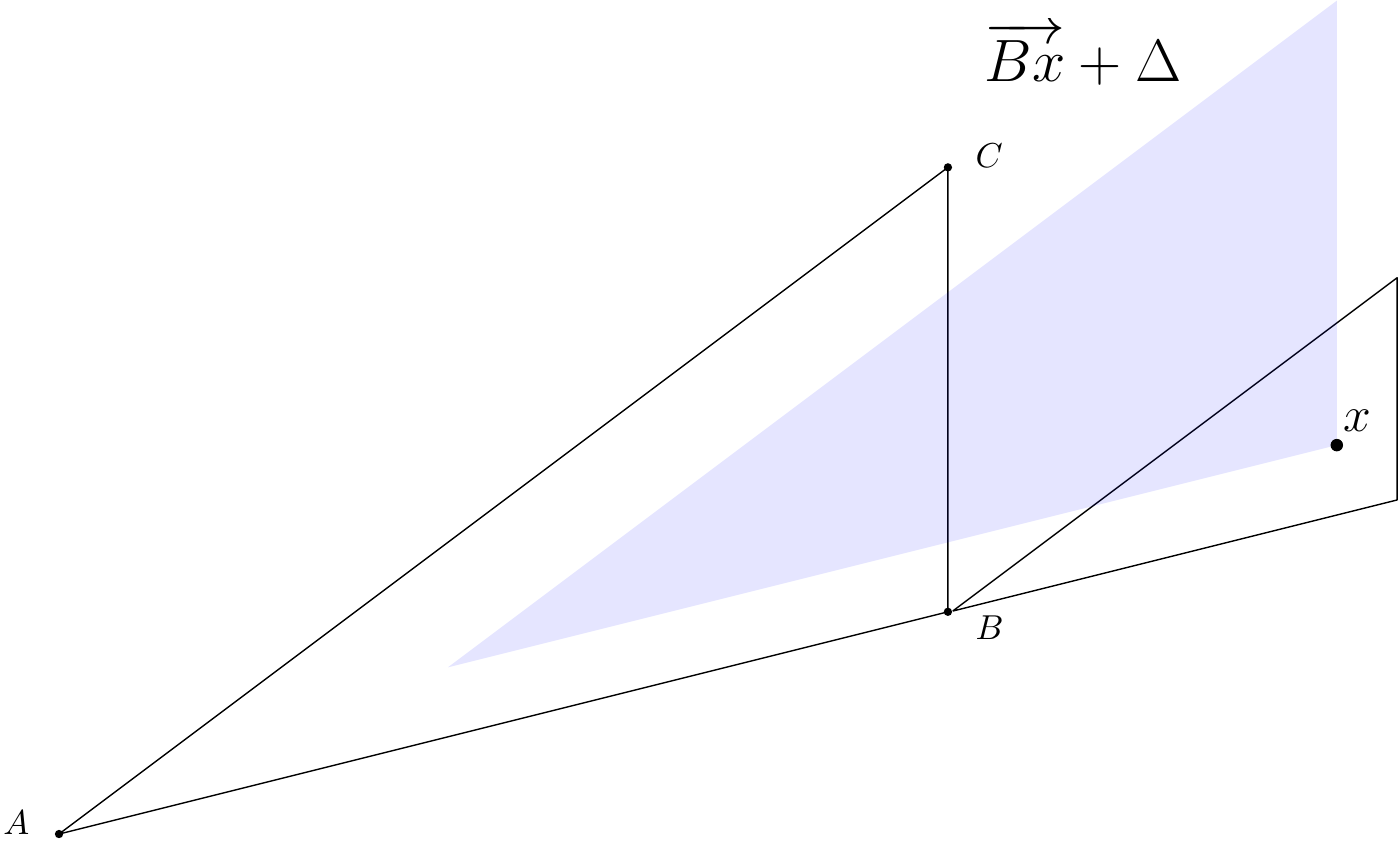}
\caption{The proof of lemma \ref{L : geo estimate 1} relies on the fact that $|\Delta \cap \left(\protect\overrightarrow{Bx} + \Delta\right)  | \geq \frac{1}{4}|\Delta|$.}
\end{figure}

\begin{proof}
Fix $x \in \Vec{B} + \frac{1}{2}(\Delta-\Vec{A})$. It suffices to observe that we have $x \in \overrightarrow{Bx} + \Delta$ and that $| \Delta \cap ( \overrightarrow{Bx} + \Delta ) | \geq \frac{1}{4} |\overrightarrow{Bx} + \Delta |$. Hence $x \in \left\{ M_{\left\{\Delta\right\}}\mathbb{1}_{\Delta} \geq \frac{1}{4} \right\}$.
\end{proof}

Actually we need a more general version of the previous estimate. For $e\ \in \mathbb{R}_+$ and $\Delta = ABC$ as before, define the triangle $T$ as $$T := T(e,\Delta) = AB(B+e\overrightarrow{BC}).$$

\begin{figure}[h!]
\centering
\includegraphics[scale=0.6]{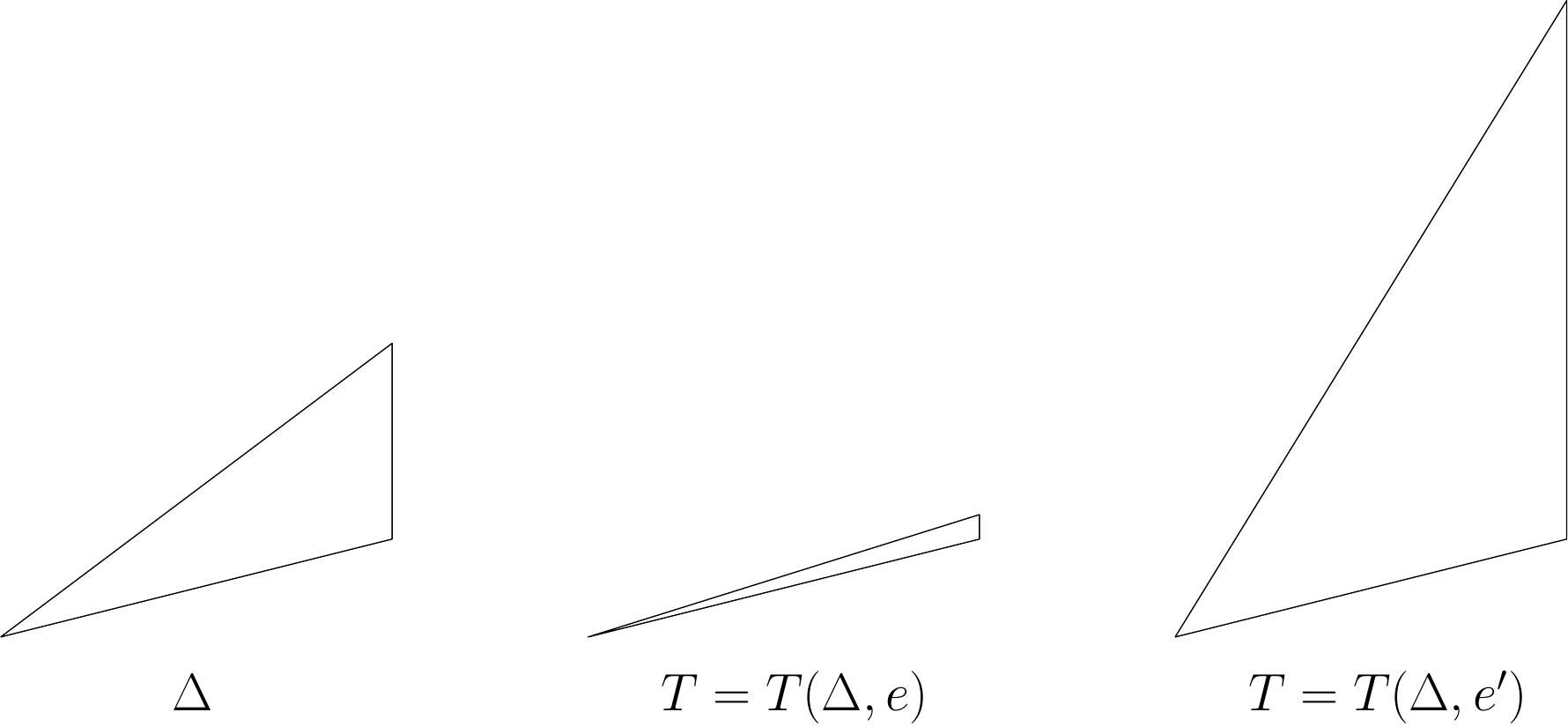}
\caption{A representation of $\Delta$ and $T=T(\Delta,e)$ for $e \ll 1$ and $e' > 1$.}
\end{figure}

\begin{figure}[h!]
\centering
\includegraphics[scale=0.7]{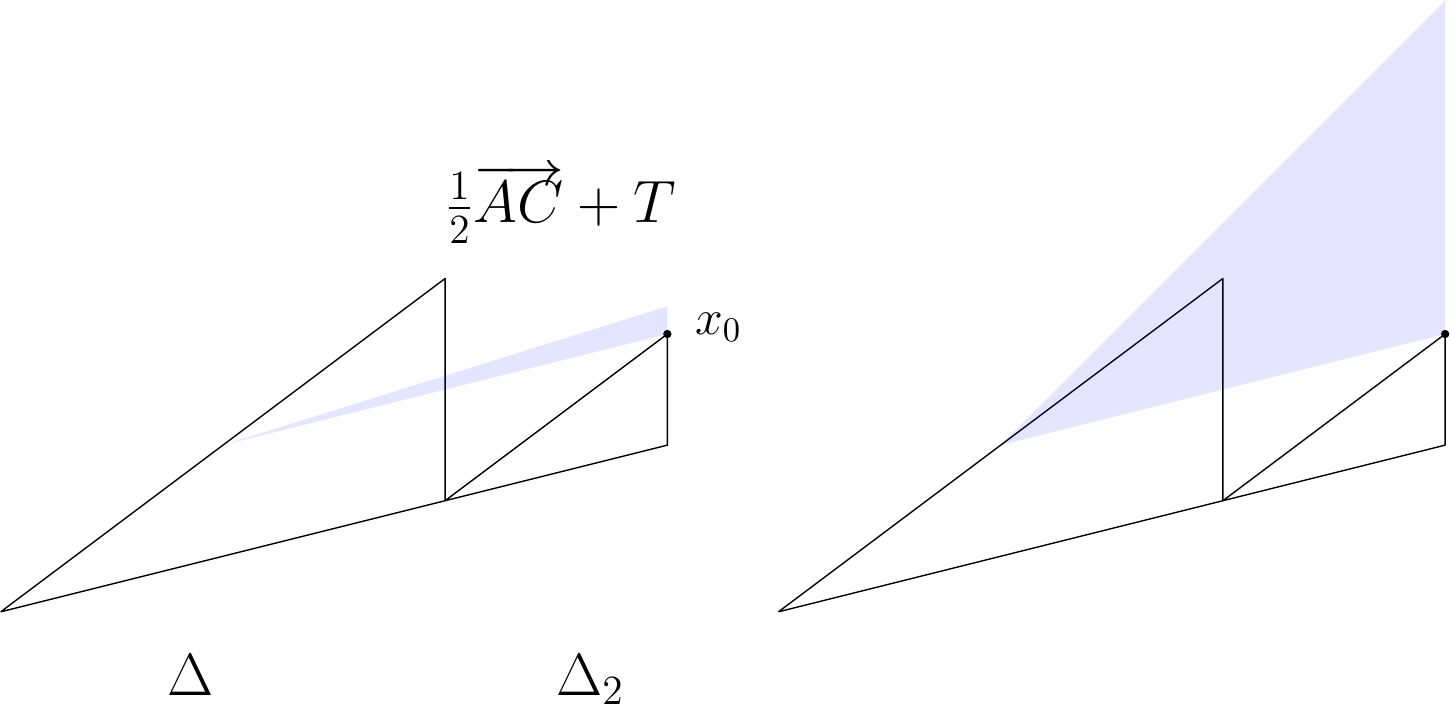}
\caption{An illustration of the argument of lemma \ref{L : estimation II} ; the left side represents the case $0 < e \leq 1$ and the right side the case $e > 1$ ; the triangle in shaded blue represents $\frac{1}{2} \protect\overrightarrow{AC} +  T$.}\label{FIG : LEMMA 2}
\end{figure}

\begin{lemma}[Geometric estimate II]\label{L : estimation II} For any pair $(\Delta,T)$ as defined above, the following inclusion holds $\Delta_2 \subset \left\{ M_{ \left\{T\right\} } \mathbb{1}_{\Delta} \geq \eta({e}) \right\}$ where $ \eta(e) = \inf \left\{ \frac{1}{4} , \frac{1}{4e} \right\} $.
\end{lemma}

\begin{proof}
The proof is akin to the proof of lemma \ref{L : geo estimate 1} and we invite the reader to look at figure \ref{FIG : LEMMA 2} for a geometric representation. It is enough to check $$x_0 \in  \left\{ M_{ \left\{T\right\}} \mathbb{1}_{\Delta} > \eta(e) \right\} $$ where $x_0 = B + \frac{1}{2}  \overrightarrow{AC} $ because this is the worst case. To begin with, observe that we have $x_0 \in \frac{1}{2} \overrightarrow{AC} +  T$. We distinguish then two situations ; if we have $$0 < e \leq 1 $$ we claim that we are in the situation corresponding to the left situation in figure \ref{FIG : LEMMA 2}, that is to say we have $$ \left|\Delta \cap  \left( \frac{1}{2} \overrightarrow{AC} +  T \right)\right| = \frac{1}{4} |T|$$ and hence also $$x_0 \in \left\{ M_{ \left\{T\right\} } \mathbb{1}_{\Delta} \geq \frac{1}{4} \right\}.$$ The second situation corresponds to the case where $1 < e $ ; in this case, we have (see figure \ref{FIG : LEMMA 2} again) $$ \left|\Delta \cap  \left( \frac{1}{2} \overrightarrow{AC} +  T \right) \right| \geq \frac{1}{4}\left|\Delta \right| \geq \frac{1}{4e} |T|.$$ This shows that we have $x_0 \in \left\{ M_{ \left\{T\right\} } \mathbb{1}_{\Delta} > \frac{1}{4e} \right\}$ and concludes the proof.
\end{proof}

\subsection*{Generalized Perron trees}\label{S:gpt}

Denote by $\Delta_k$ the triangle whose vertices are the points $O, A_k = (1, t_k)$ and $A_{k+1} = (1, t_{k+1})$. Recall that we have supposed $\tau_{\boldsymbol{t}} < \infty$. We now give a slighlty improved version of the construction of generalized Perron trees as defined in \cite{KATHRYN JAN}

\begin{figure}[h!]
\centering
\includegraphics[scale=0.8]{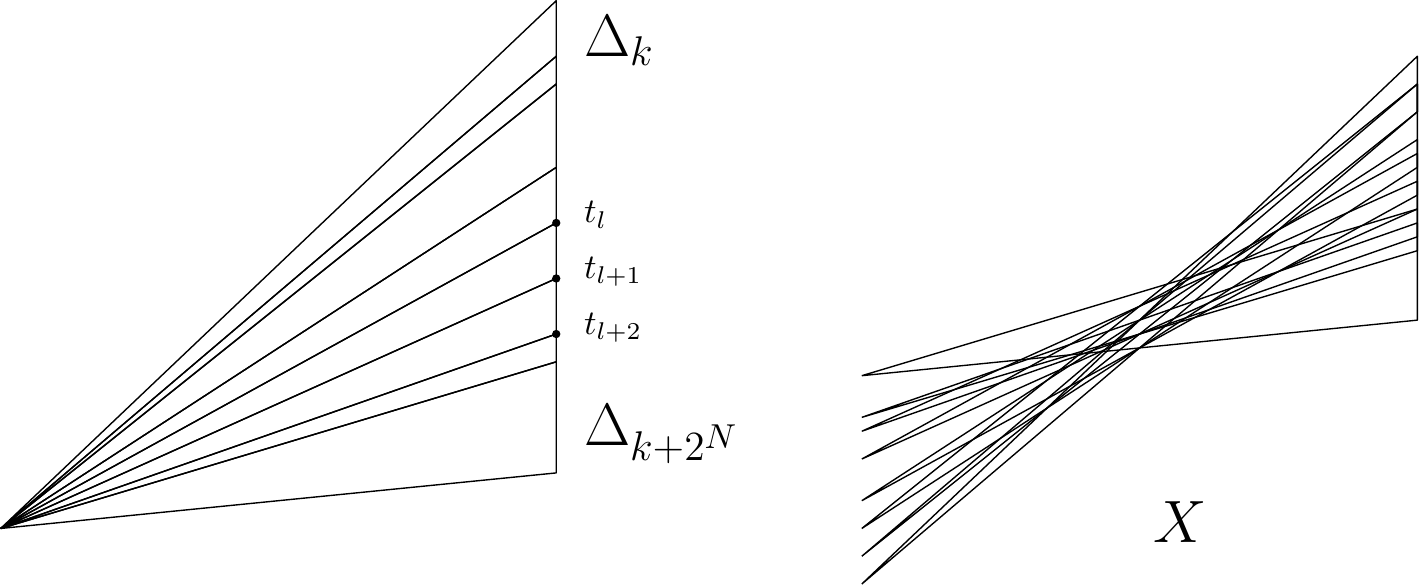}
\caption{ A representation of some $\Delta_k$ and on the left side a Perron tree $X$ generated with those triangles. The idea is that for large $n$ one has $|X| \ll |\Delta_k \sqcup \dots \sqcup \Delta_{k+2^N}|$ ; plus the second property of proposition \ref{ P : gpt }.}
\end{figure}

\begin{prp}[Generalized Perron Tree]\label{ P : gpt }
For any positive ratio $\alpha$ close to $1$ and any integer $n \geq 1$, there exists an integer $N \gg 1$ and $2^n$ vectors $\Vec{s}_k := (0,s_k)$ such that defining the set $$X = \bigcup_{N+1 \leq  k \leq N + 2^n} \left( \Vec{s}_k + \Delta_k \right)$$ we have the following properties \begin{itemize}
    \item  $ |X| \leq \left( \alpha^{2n} + \tau_{\boldsymbol{t}}(1-\alpha) \right) \left|\Delta_{N+1} \sqcup \dots \sqcup \Delta_{N+2^n}\right|$;
    \item for any $k \neq l$ the triangles $\left( \Vec{A}_k + \Vec{s}_k \right)+ \frac{1}{2}\Delta_k$ and  $\left( \Vec{A}_l + \Vec{s}_l \right)+ \frac{1}{2}\Delta_l$ are disjoint.
\end{itemize} We say that the set $X$ is a \textit{generalized Perron tree of scale $(\alpha,n)$} and we denote it by $ X_{\alpha,n}(\boldsymbol{t})$.
\end{prp}

Note that the fact that the triangles $ \left( \Vec{A}_k + \Vec{s}_k \right)+ \frac{1}{2}\Delta_k$ and  $ \left( \Vec{A}_l + \Vec{s}_l \right)+ \frac{1}{2}\Delta_l$ are disjoint is not proven in \cite{KATHRYN JAN} yet it is a simple observation. Observe that for any $\epsilon > 0$, one can first choose $\alpha$ close to one and then $n$ large enough in order to have $$|X_{\alpha,n}(\boldsymbol{t})| \leq \epsilon  \left|\Delta_{N+1} \sqcup \dots \sqcup \Delta_{N+2^n}\right|$$ for some large $N$. To obtain such an inequality, we need a sufficient condition on the thin triangles $\Delta_k$ that ensures in some sense that they are comparable. Indeed, suppose that we had defined for any $k \geq 1$ the triangle $\Delta_k$ as the one whose vertices are the points $O,  G_k = (1, \frac{1}{2^k})$ and $G_{k+1} = (1,\frac{1}{2^{k+1}})$. In this situation, for any $I \subset \mathbb{N}$ and any sequence of vectors $\left\{ \Vec{s}_i \right\}_{i \in I} \subset \mathbb{R}^2$ the set $X_I$ defined as $$ X_I = \bigcup_{i \in I}  \left( \Vec{s}_i + \Delta_i \right) $$ satisfies the following inequality $$ |X_I| \geq |\Delta_{i_0}| \geq \frac{1}{2} \left|\bigcup_{i \in I} {\Delta_i}\right|$$ where $i_0 := \min I$. Hence we cannot hope to stack up the triangles $\Delta_k$ into a set $X$ that has a small area compared to the sum of the areas of the $\Delta_k$. Hopefully this example shed light on the condition imposed on $\boldsymbol{t}$ which is  $$ \tau_{\boldsymbol{t}} := \sup_{k \geq 0, l \leq k} \left( \frac{t_{k+2l} - t_{k+l}}{t_{k+l} - t_{k}} + \frac{t_{k+l} - t_{k}}{t_{k+2l} - t_{k+l}} \right) < \infty.$$ This ensure that the triangles $\Delta_k$ are comparable in some sense and that we can construct \textit{generalized Perron trees} with them.

\begin{figure}[h!]
\centering
\includegraphics[scale=0.8]{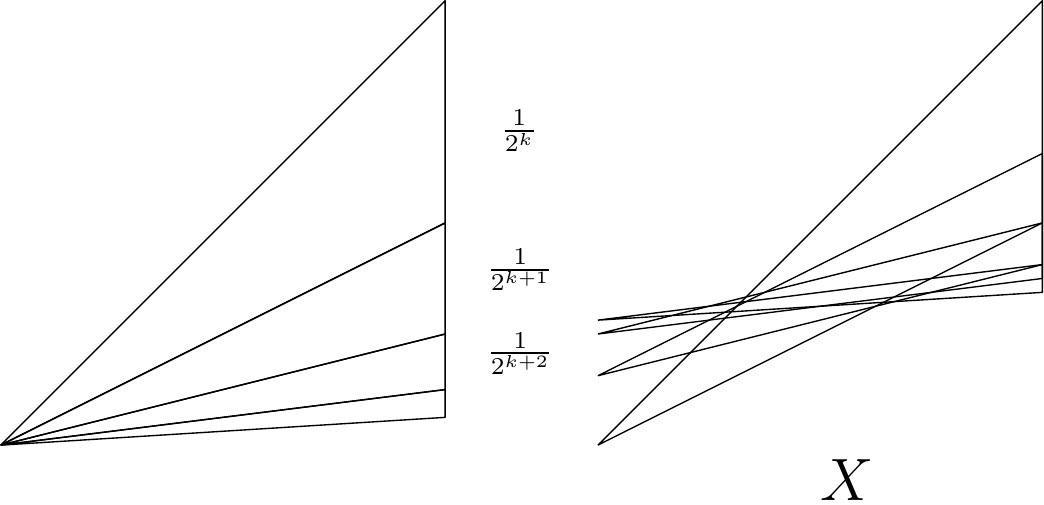}
\caption{It quite difficult to construct a Perron tree ; one needs a condition to ensure that the triangles $\Delta_k$ are \textit{comparable} in some sense. On this figure, the $\Delta_k$ differs too much in volume and one will always have $|X| \simeq |\cup \Delta_i|$ as explained.}
\end{figure}


\subsection*{Proof of Theorem \ref{T : main 2}}

Recall that we suppose there is a constant ${\mu_0} > 0$ such that for any $k \geq 1$, $e_k < {\mu_0} |t_k-t_{k+1}|$. To begin with, we are going to construct a Perron tree $X_{\alpha,n}(\boldsymbol{t})$ with the triangles $\left\{ \Delta_k \right\}_{k \geq 1}$. Then we will exploit this Perron tree $X_{\alpha,n}(\boldsymbol{t})$ with the triangles $\mathcal{B} = \left\{ T_k \right\}_{k \geq 1}$ to show that $M_\mathcal{B}$ is a bad operator. Precisely we prove the following claim.

\begin{claim}
For any $\alpha$ close  to $1$ and any $n \in \mathbb{N}$, the Perron tree $X := X_{\alpha,n}(\boldsymbol{t})$ satisfies the following inequality 
   $$ |X| \leq  \epsilon |\left\{ M_\mathcal{B}\mathbb{1}_{X} > \eta({\mu_0}) \right\}|$$ where $\epsilon = \alpha^{2n} + \tau_{\boldsymbol{t}}(1-\alpha)$.
\end{claim}

\begin{proof}
Fix $\alpha$ close to $1$ and  $n \in \mathbb{N}$ and consider a Perron tree of scale $(\alpha,n)$ $$ X := X_{\alpha,n}(\boldsymbol{t}) = \bigcup_{N+1 \leq  k \leq N + 2^n} \left( \Vec{s}_k + \Delta_k \right)$$ where $N$ is given by proposition \ref{ P : gpt }. Fix any $k \in \{N+1, \dots,N + 2^n\}$ and consider the pair of triangles  $$(\Vec{s}_k + \Delta_k,  \Vec{s}_k + T_k)$$ or more simply the pair $(\Delta_k,T_k)$ which is the same up to a translation. We can apply lemma \ref{L : estimation II} to this pair which yields the following inclusion $$ \left(\Vec{A}_{k+1} + \Vec{s}_k \right) + \frac{1}{2}\Delta_k \subset \left\{ M_{ \left\{ T_k\right\} } \mathbb{1}_{ \Vec{s}_k + \Delta_k } > \eta(\mu_0) \right\}.$$ Since we have $M_{T_k} \leq M_\mathcal{B}$ we also have $$\left(\Vec{A}_{k+1} + \Vec{s}_k \right) + \frac{1}{2}\Delta_k \subset \left\{ M_\mathcal{B} \mathbb{1}_{ \Vec{s}_k + \Delta_k } > \eta(\mu_0) \right\}.$$ The previous inclusion then yields $$\bigsqcup_{k = N}^{N + 2^{n}}   \left(\Vec{A}_{k+1} + \Vec{s}_k \right) + \frac{1}{2}\Delta_k \subset \left\{M_\mathcal{B}\mathbb{1}_X > \eta({\mu_0})  \right\}.$$ In the latter inclusion, the fact that the union is disjoint comes from our statement of Proposition \ref{ P : gpt }. Hence this gives in terms of Lebesgue measure $$\sum_{N +1 \leq  k  \leq N + 2^n} \frac{1}{4}|\Delta_k| \leq | \left\{M_\mathcal{B}\mathbb{1}_X > \eta({\mu_0})  \right\} |.$$ Using the fact that $X$ is a Perron tree constructed with the triangles $\Delta_k$ we have $$|X| \leq \left( \alpha^{2n} + \tau_{\boldsymbol{t}}(1-\alpha) \right) |\Delta_{N+1} \sqcup \dots \sqcup \Delta_{N+2^n}|.$$ In other words we have $$|X| \leq  4 \left( \alpha^{2n} + \tau_{\boldsymbol{t}}(1-\alpha) \right) |\left\{M_\mathcal{B}\mathbb{1}_X > \eta({\mu_0})  \right\} |.$$

\end{proof}

Observe finally that the claim implies that for any $p > 1$ we have  $$\|M_\mathcal{B}\|_p \geq \eta({\mu_0}) ( 4\alpha^{2n} + 4\tau_{\boldsymbol{t}}(1-\alpha) )^{-\frac{1}{p}} $$ for any $\alpha$ close to $1$ and any $n \in \mathbb{N}$. The fact that constant $\eta({\mu_0})$ is independant of the scale $(\alpha,n)$ concludes : we have $ \|M_\mathcal{B}\|_p = \infty$ for any $p > 1$ \textit{i.e.} $M_\mathcal{B}$ is a bad operator, and thus Theorem \ref{T : main 2} is proved.

\section{Proof of Theorem \ref{THMA}}

\begin{figure}[h!]
\centering
\includegraphics[scale=0.8]{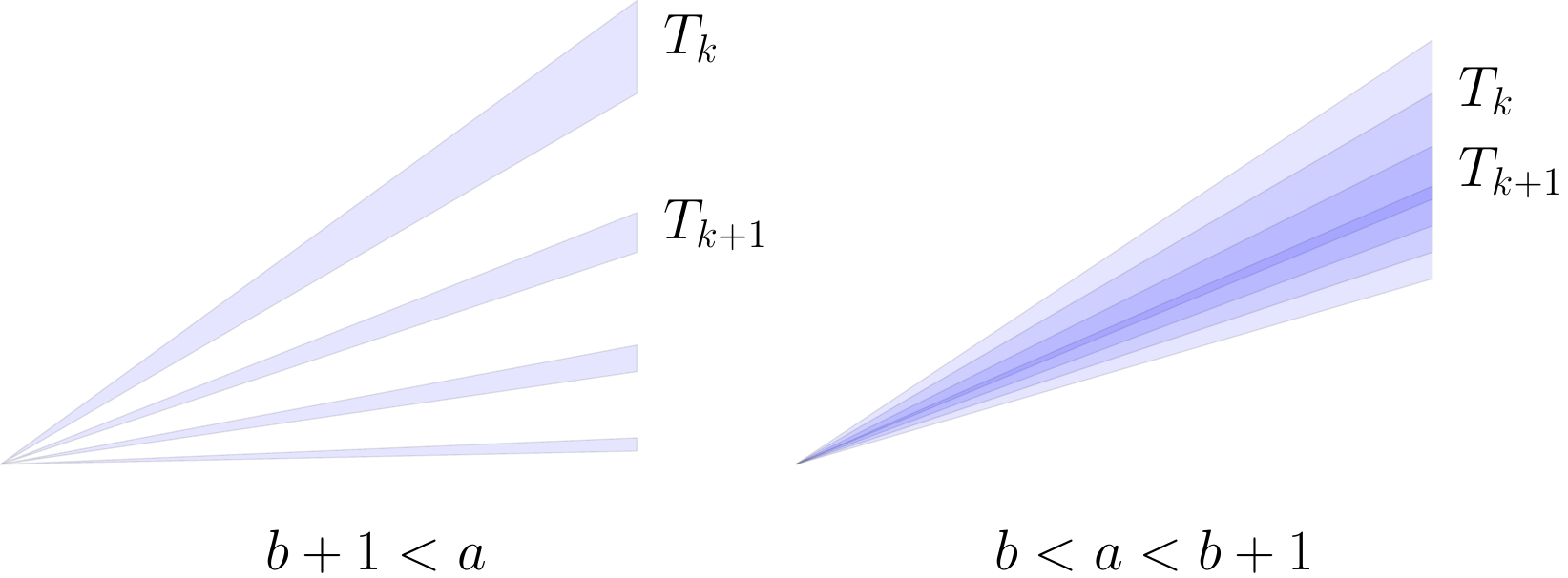}
\caption{On the left side a representation of the regime $a > b+1$. In this situation, the triangles $T_k$ do not overlap at all for large $k$ (actually the gap gets bigger with $k$). On the right side a representation of the regime $ b+1 > a > b$. In this situation, the triangles $T_k$ tend to completely overlap each other.}\label{FIG : strategy}
\end{figure}




We are now ready to prove Theorem \ref{THMA} ; we fix $a,b > 0$ and recall that we define the basis $\mathcal{B}_{a,b}$ as the one generated by a sequence of rectangles $\left\{ R_n \right\}_{n \geq 1}$ satisfying $$ \left(e_{R_n}, \omega_{R_n} \right) = \left( \frac{1}{n^a} , \frac{\pi}{4n^b} \right).$$ Recall also that for any $b > 0$ letting $\boldsymbol{\omega} = \{ \frac{\pi}{4n^b} \}$ we have $\tau_{\boldsymbol{\omega} } < \infty$.

\subsection*{Case $a \leq b$}

In the case $a \leq b$, observe that we have for $n \geq 1$ $$ \frac{4}{\pi n^b} \lesssim \frac{1}{n^a} $$ and so applying Theorem \ref{T : main 1} we obtain that $M_\mathcal{B}$ is a good operator.

\subsection*{Case $a \geq b+1 $}

In the case $a \geq b+1 $, observe that we have for $ n \geq 1$ $$ \left| \frac{\pi}{n^b} - \frac{\pi}{(n+1)^b} \right| \simeq \frac{1}{n^{b+1 }} $$ and so we have $$ \frac{1}{n^a} \lesssim  \left| \frac{\pi}{n^b} - \frac{\pi}{(n+1)^b} \right|.$$ We can apply Theorem \ref{T : main 2} which implies that $M_\mathcal{B}$ is a bad operator.

\subsection*{Case $b < a < b+1$}

Observe that for any $\ell \in \mathbb{N}^*$, we have $\mathcal{B}_{ \ell a , \ell b } \subset \mathcal{B}_{a,b}$. Hence we trivially have $$M_{\ell a, \ell b } \leq M_{a,b}.$$ Since $a > b$, for $\ell_0 \gg 1$ we have $a > b + \frac{1}{\ell_0}$ that is to say $$\ell_0 a > \ell_0 b +1.$$ Applying the previous case, it appears that $M_{\ell_0a, \ell_0b}$ is a bad operator and so is $M_{a,b}$.

{}

\end{document}